\documentclass[12pt]{article}

\usepackage{amsmath}
\usepackage{amsfonts}
\usepackage{amssymb}
\usepackage{amsthm} 
\usepackage{verbatim}
\usepackage[left=2.9cm,right=2.9cm,top=3cm,bottom=3cm]{geometry}
\usepackage[dvipsnames]{xcolor}
\usepackage{graphicx}
\usepackage{enumerate}

\usepackage[affil-it]{authblk} 

\newcommand{\mP}{{\mathbb P}}

\newcommand{\G}{\mathcal{G}}

\newcommand{\N}{\mathbb N}

\newtheorem{theorem}{Theorem}
\newtheorem{lemma}{Lemma}
\newtheorem{rem}{Remark}

\newtheorem{example}{Example}
\newtheorem{definition}{Definition}
\newtheorem{corollary}{Corollary}

\title{Voting Profiles Admitting All Candidates as Knockout Winners}

\author[1]{Bernard De Baets}
\author[2]{Emilio De Santis}

\affil[1]{\small KERMIT, Dept. of Data Analysis and Mathematical Modelling, Ghent University, Coupure links 635, Ghent 9000, Belgium. \textit{bernard.debaets@ugent.be}}
\affil[2]{\small Department of Mathematics, Sapienza University of Rome, Piazzale Aldo Moro 5, 00185, Rome, Italy. \textit{desantis@mat.uniroma1.it}}

\date{} 

\begin{document}

\maketitle

\begin{abstract}
A set of $2^n$ candidates is presented to a commission. At every round, each member of this commission votes by pairwise comparison, and one-half of the candidates is deleted from the tournament, the remaining ones proceeding to the next round until the $n$-th round (the final one) in which the final winner is declared. The candidates are arranged on a board in a given order, which is maintained among the remaining candidates at all rounds. A study of the size of the commission is carried out in order to obtain the desired result of any candidate being a possible winner. For $2^n$ candidates with $n \geq 3$, we identify a voting profile with $4n -3$ voters such that any candidate could win simply by choosing a proper initial order of the candidates. Moreover, in the setting of a random number of voters, we obtain the same results, with high probability, when the expected number of voters is large. 
\vspace{1em}


\end{abstract}

\section{Introduction}

In this paper we consider a set of \emph{candidates} (or competitors) that are involved in a knockout tournament. In the scientific literature, such tournaments, also called \emph{single-elimination or sudden death tournaments}, are studied for different practical and theoretical purposes, see, e.g.,~\cite{kim2015fixing,MS2023,Mossel-Xu2020,W2010, Williams2016} and citations therein. A critical drawback of knockout tournaments is the fact that the winner of the tournament could heavily depend on the initial order of the candidates, also called \emph{bracket}. In fact, tournament organizers could favor a given candidate by selecting a particular bracket; see, e.g.,~\cite{aziz18, MS2023, Marchand2002} and references therein.

In the setting considered in this paper, there is a commission of voters and each voter ranks all candidates in a \emph{preference list} (no ties allowed), i.e., each voter provides a permutation of all candidates.  The candidates are arranged on a board in a given order, with the relative order among them maintained at all rounds. Running through this order, candidates are presented in pairs to the commission, which decides by majority which of the two candidates wins the match and is allowed to move on to the next round. To be precise, each voter only looks at her own preference list and votes for the candidate who precedes the other one. It should be clear that the winner of the competition is a function of the preference lists of all voters, called \emph{voting profile}, and the bracket. 
In the literature on voting theory, 
the \emph{preference pattern} or \emph{majority graph} is constructed as a function of the voting profile and establishes the results of matches between any two candidates; however, the majority graph does not give the number of votes received by the two candidates. 

Manipulation in voting theory involves strategies used by voters or participants to sway the results of an election or decision-making process in their favor, often compromising fairness and genuine preferences. This issue is a key focus in social choice theory, which examines how collective decisions are made and how they can be influenced by strategic actions. The Gibbard--Satterthwaite theorem, see, e.g.,~\cite{Gib1973,isaksson2012geometry,sen2001}, illustrates that in most voting systems, some form of manipulation is inevitable when participants are rational and have incentives to influence results. 

For knockout tournaments, there is a different form of influence, coming from the organizers, called \emph{agenda problem}: the problem of strategically controlling the structure of the tournament (the schedule or agenda) to influence the outcome (or maximize the winning probability) of a specific player or a group of players. This involves designing the tournament bracket in a way that favors certain candidates and disadvantages other ones.  
Elimination-style elections often resemble knockout tournaments. These knockout tournament formats, which have been explored in various decision-making contexts, have garnered significant attention in fields such as artificial intelligence, see, e.g.,~\cite{aziz18,W2010,vu2009}, economics and operations research, see, e.g.,~\cite{Conitzer16, laslier97,rosen}.  

In real life, there are numerous examples of elimination-style elections similar to knockout tournaments. Although these formats are less common in formal political elections, similar elimination methods are employed in primary elections and candidate selection processes, particularly during party conventions or leadership contests. For instance, in certain political primaries, candidates may be gradually eliminated based on vote thresholds. Additionally, photography competitions sometime adopt a knockout format, where members vote on pairs of photos, allowing the preferred image to advance to the next round; see, e.g., Knock-Out PDI Competition 2025\footnote{See https://beaulieucameragroup.co.uk/news/407-knock-out-pdi-competition-2025}. Singing competitions also frequently utilize similar knockout structures\footnote{Sanremo Music Festival 2024: three young singers were selected through three knockout tournaments, each consisting of four young contestants, and then advanced to the main competition.}.

The central questions of this paper are the following: 
\begin{itemize}
\item[\bf{Q1}] Does there exist a majority graph that admits each candidate as winner, provided one chooses a proper bracket? 
\item[\bf{Q2}] How large must the \emph{voting profile} be, i.e.,~the set of preference lists of a commission, that admits each candidate as winner? 
\end{itemize}
From Theorem~6 in~\cite{W2010}, we can answer {\bf{Q1}} affirmatively, at least for large numbers of candidates (see also~\cite{KSV17,StaVas11}).
However, in view of {\bf{Q2}}, we need an explicit construction that helps to keep the number of voters low.

To that end, we answer the first question in a constructive manner, confirming that it is indeed feasible to construct majority graphs that allow any candidate to emerge victorious through careful selection of the tournament bracket. For the second inquiry, we establish an upper bound on the necessary size of the voting profile, which is logarithmic in the number of candidates. 


The solutions are constructed in two steps. In the first step, we construct a sequence of \emph{preference patterns} or \emph{majority graphs}, denoted $(G_n)_{n \geq 3}$, which establishes the results of the matches between any two candidates and answer affirmatively to question {\bf{Q1}}. In the second step, 
we construct the voting profiles $(\tilde{R}_n)_{n \geq 3}$ such that $\tilde{R}_n$ gets $G_n$ as majority graph. Regarding the first step, we construct the majority graphs $(G_n)_{n \geq 3}$ recursively. The fundamental property of $G_n$ is that, in a knockout tournament with a number of candidates equal to $2^{n}$, any candidate can win given an appropriate bracket.

Moving to the second step, it is well known that any majority graph can be obtained by a simple majority decision of individuals with a suitable voting profile~\cite{ErdosMoser,McGarvey1953, Stearns}. Thus, there exists a voting profile $\tilde{R}_n$ associated with the majority graph $G_n$. For a construction of $(\tilde{R}_n)_{n \geq 3}$ with a small number of voters,    we proceed by recursion. The initial voting profile~$\tilde{R}_3$ is realized through Stearns' paper~\cite{Stearns}. For the recursion, to pass from~$\tilde{R}_n$ to~$\tilde{R}_{n+1}$, we make use of Lemma~\ref{EMnew} below and Lemma~2 of~\cite{ErdosMoser}. Following our approach, the voting profile~$\tilde{R}_{n+1}$ equals the size of~$\tilde{R}_n$ plus 4. Since~$\tilde{R}_3$ consists of nine preference lists, the voting profile $\tilde{R}_n$ corresponds to $4n-3$ preference lists. 

In the last part of this paper, we consider as an example the framework of a random number of voters (see, e.g.,~\cite{CS23,FP99, makris2008}). Under the hypothesis that the number of voters follows a Poisson distribution, we establish that the random voting profile behaves like the deterministic one, with high probability, if the Poisson parameter is sufficiently large.

This paper is structured as follows. In Section~\ref{Notation}, we give some notation and definitions. In Section~\ref{sec2bis}, we construct the majority graphs $(G_n)_{n \geq 3 }$. Moreover, in Theorem~\ref{albero}, we show that a knockout tournament with $2^{n}$ candidates (with $n \geq 3$) and a majority graph $G_n$ can return any candidate as the winner of the tournament by choosing an appropriate bracket. In Section~\ref{construction voting}, we construct the voting profile $\tilde{R}_n$ associated with the majority graph~$G_n$, for $n \geq 3$. In Section~\ref{sec:random}, we study an example of random voting profiles in the
Poisson framework. Finally, in Section~\ref{conc}, we present some comments and possible future developments.

\section{Notations and definitions} \label{Notation}
We start this section by introducing some notations that will be used further on. Let $[m]:=\{1,\ldots,m\}$. We will consider only the case $m = 2^n$, for any $n \in \N$. Let $\Pi(n)$ denote the set of permutations of the elements in $[2^n]$. 

All $2^n$ candidates are initially arranged on a board in a given order (a permutation $\boldsymbol{\pi} \in \Pi(n)$). Then, pairwise matches are carried out, following the initial order to form the pairs. After the first round of the matches, all winners are paired again, still respecting the initial order, and so on, until the winner of the competition is decreed at the $n$-th round. To determine the winner of a match between two candidates, there is a commission of $v \in \N$ voters in which each member assigns a vote to a candidate, the candidate getting the most votes being the winner of the match. Each member of the commission has a fixed preference list in which all candidates are written in the order of preference (no ties are allowed). The preference list of a voter corresponds to a permutation in $\Pi(n)$. When two candidates $i,j \in [m]$ are paired, each voter gives her vote to the preferred candidate, that is, she votes for $i$ if $i$ precedes $j$ in her preference list; otherwise she votes for $j$. 

The collection of all the preference lists of the voters of the commission is commonly called the \emph{voting profile}. Suppose that $V$ is the set of candidates, typically $V =[m]$, and $[v]$ the set of voters, then the voting profile is represented as a matrix $R= (r_{\ell,i}: \ell \in [v], i \in V)$, where $r_{\ell, \cdot}$ denotes the preference list of voter $\ell \in [v]$, i.e. the row $r_{\ell, \cdot}$ is a permutation of the elements of $V$.

For the reader's convenience, we recall the standard definition of a directed graph that will be used in this article.

\begin{definition}
A directed graph $G = (V, E)$ is an ordered pair, where
$V$ is a set whose elements are called \emph{vertices}. Furthermore, 
$E$ is a set of ordered pairs of vertices, called \emph{arrows}.
\end{definition}

On the basis of the voting profile $R$, we construct the \emph{majority graph} (also called \emph{preference pattern}).

\begin{definition}[Majority graph] \label{Mgraph}
  For any voting profile $R$ on a set of candidates $V$, the majority graph $G_R$ is a directed graph, such that its set of vertices is the set of candidates $V$ and for every two candidates $a,b \in V$, $(a,b)$ is an arrow in this graph if and only if candidate $a$ is preferred over candidate $b$ by more than half of the voters. We also write that $R$  generates the directed graph $G$ when $G=G_R$.
\end{definition} 

In the following, we will use the concepts of \emph{induced directed graph} and \emph{isomorphism between directed graphs}.

\begin{definition}[Induced directed graph]\label{indu}
Let ${\displaystyle G=(V,E)} $ be a directed graph and let 
$S\subseteq V$ be any subset of $V$. Then the induced directed graph
$G[S]$ is the directed graph whose vertex set is
$S$ and $(a,b)$ is an arrow of $ G[S]$ if and only if $(a,b ) \in E $ and $a,b \in S$. 
\end{definition}

\begin{definition}[Isomorphism between directed graphs]\label{iso}
Let $G_1=(V_1,E_1)$ and $G_2=(V_2,E_2)$ be two directed graphs. We say that $G_1$ and $G_2$ are isomorphic if there exists a bijection $f: V_1 \to V_2 $ such that 
\[
(u,v)  \in E_1 \iff (f(u), f(v)) \in E_2\,.
\]
\end{definition}

Next, we provide some additional notations. For a sequence $(R_n)_{n \in [N]}$ of matrices with the same number of columns $m$, we denote by $(R_1, \ldots , R_N)_V$  the \emph{vertically glued matrix}, i.e.,~the matrix that has as rows all rows of matrices $R_1, \ldots , R_N$. Note that if the matrices $R_1, \ldots , R_N$ are voting profiles on a common set of candidates, then $(R_1, \ldots, R_N)_V$ represents a new voting profile on the same set of candidates, namely the collection of all the preference lists reported in the different matrices $R_1, \ldots , R_N$. For a sequence $(v_n)_{n \in [N]}$ of (row) vectors we define the \emph{concatenated vector} $(v_1, \ldots , v_N)$ as the vector with all elements of $v_1, \ldots , v_N$ listed sequentially. Given a vector $v = (v_1, v_2, \ldots, v_n)$, we denote by $\bar v$ the vector $(v_n, v_{n-1}, \ldots, v_1)$ with all components of $v $ in reverse order.

\section{Construction of majority graphs}
\label{sec2bis} 
We have explained above how to construct the majority graph starting from a voting profile. The collection of finite directed graphs without loops is denoted by $\G$. It is well known that every $G \in \G$ can be generated by some voting profile~$R$ (see~\cite{McGarvey1953}, \cite{Stearns} and~\cite{ErdosMoser}, in historical order). 

Therefore, for any $G =(V, E)\in \G$, there exists a voting profile $R$ that generates the results prescribed by $G$ for all pairs
of candidates in $V$. Actually, in the following, we will consider preference patterns without ties, called \emph{strict} preference patterns; such preference patterns are nothing else but complete directed graphs. For the set of candidates~$[2^n]$, the strict preference patterns (or majority graphs) are collected in the space~$\mathcal{H}(n)$. 

Note that given a voting profile $R$ and the bracket $\boldsymbol{\pi}\in \Pi (n)$, one can list all the matches that will take place in the knockout tournament along with their outcomes. Furthermore, in order to reconstruct all the results of the knockout tournament, it suffices to know the initial order $\boldsymbol{\pi}\in \Pi (n)$ and the majority graph $G \in \mathcal{H}(n)$ associated with~$R$. In fact, even if $G$ cannot provide the number of votes obtained by the two candidates in a match, it provides the result of the match itself. 

In particular, the winner of the competition is a function of the majority graph~$G \in \mathcal{H}(n)$ and of the initial order $\boldsymbol{\pi} \in \Pi (n)$. Thus, we define the mapping
\[
w_n: \mathcal{H}(n) \times \Pi (n) \to [ 2^n]
\]
which returns the winner of the game with $2^n$ candidates, given the majority graph and the initial order of the candidates.

For any majority graph $G\in \mathcal{H} (n)$, we consider the
image of $w_n (G, \cdot)$, i.e., 
\[
w_n  (G, \Pi (n)  ) = \{ w_n  (G, \boldsymbol{\pi} ) :  
 \boldsymbol{\pi} \in \Pi (n) \}\,. 
\]
We are interested in establishing for which values of  $n \in \mathbb{N}$ there exists a 
majority graph $G \in \mathcal{H} (n)$ such that   
\begin{equation}\label{tutti}
w_n  (G, \Pi (n)  ) = [2^n]\,,
\end{equation}
i.e., all the candidates can win just by changing the initial order on the board. Along the paper, the majority graphs and the brackets of interest will be explicitly constructed. 

We start  by defining the  majority graph $G_3=
 (V_3, E_3) \in \mathcal{H} (3)$  where  $V_3=[8]$ and 
all the  arrows of  $E_3$   are listed below
\begin{equation} \label{pref}
\begin{array}{c}
(1, 2)  , \qquad (1 , 3 ), \qquad (1 ,  5) ,\qquad (2 ,  3 ),
 \qquad (2 ,  4 ), \qquad (2 , 6) , \qquad (2 ,  7) ,   \\
(3 , 4 ), \qquad (3 ,  5) , \qquad (3 ,  6 ),\qquad (3,  8) , \qquad (4 ,  1) , \qquad (4 ,  5), \qquad (4,  8) ,\\
(5 , 2) ,  \qquad  (5,  6) , \qquad (5 , 7),
\qquad (6, 1),  \qquad (6,  4), 
\qquad (6,  7) , \qquad (6, 8 ),  \\
(7 ,  1) ,  \qquad  (7 ,  3 ), \qquad (7 ,  4) ,
\qquad (7, 8),  \qquad (8, 1), 
\qquad (8, 2) , \qquad (8,  5). 
\end{array}
\end{equation}

The choice to start from index 3, for the first considered majority graph,  seems unnatural at this stage but in the following this  choice will avoid to subtract the index 3 in many formulas, which would make them less clear.

In the following lemma, we show that the directed graph $G_3$ admits any candidate $i \in [8]$ as a winner of the knockout tournament simply by choosing an appropriate bracket.

\begin{lemma}\label{8conc} 
Let $[8]$ be the set of candidates, then
$w_3  (G_3, \Pi (3)  ) = [8]$.
\end{lemma}

\begin{proof}   
Let us fix the initial orders of the candidates $\boldsymbol{\pi}^{(i)} \in \Pi (3)$, for $i \in [8]$, as follows 
\begin{equation} \label{base}
\begin{array}{c}
\boldsymbol{\pi}^{(1)} = (1,2,3,4,5,6,7,8)  , \quad  
\boldsymbol{\pi}^{(2)} = (2,3,4,1,6,7,8,5)   ,    \\
\boldsymbol{\pi}^{(3)} = (3,6,4,1,8,2,5,7)  , \quad  
\boldsymbol{\pi}^{(4)} = (4,1,8,2,5,6,7,3),    \\
\boldsymbol{\pi}^{(5)} = (5,6,7,8,2,3,4,1)  , \quad  
\boldsymbol{\pi}^{(6)} = (6,7,8,5,1,2,3,4),     \\
\boldsymbol{\pi}^{(7)} = (7,8,1,2,3,5,6,4)  , \quad  
\boldsymbol{\pi}^{(8)} = (8,1,5,6,2,3,7,4).
\end{array}
\end{equation}
One can easily check that $w_3(G_3,\boldsymbol{\pi}^{(k)}  ) = k$ for any $ k \in [8]$.
\end{proof}

Starting from the directed graph $G_3$, defined in~\eqref{pref}, we  construct, for any integer $n \geq 4$, the  majority graph
\[
G_n = (V_n , E_n)  \in \mathcal{H} (n) \,, 
\]
whose set of vertices $V_n $ is equal to $[2^n]$ and whose set of arrows is described below.
To construct the graph, we start with a partition of the vertices in classes. For the set of vertices $V_n = [2^n]$, we consider the partition $(C_q: q \in 2^{n-3} )$, with  
\[
  C_q =[8q] \setminus [8(q-1)]\,.
\]
We write $i \sim j $ when $i$ and $j$ belong to the same class.

For a given integer $n \geq 4 $ and $q \in 2^{n-3}$, we impose that the induced directed graph~$G_n [C_q]$ is isomorphic with the directed graph $G_3$, where the bijection $f: C_q \to V_3 (= C_1)$ is given by the following translation 
\begin{equation}\label{GI}
f(i) = i- 8(q-1), \text{ for any } i \in C_q\,. 
\end{equation}
This relation establishes all the arrows between the vertices belonging to the same class~$C_q$. Then, if $i, j \in V_n$ belong to different classes with $i <j$, the arrow $(i,j) $ belongs to $ E_n$ if and only if $i+j$ is even; otherwise $(j,i)$ belongs to $E_n$. Thus, the arrows are defined for any pair of vertices.

Consider, as an example, the graph $G_5 =(V_5, E_5)$ with $V_5 = [2^5]$ partitioned into four classes $C_1$, $C_2$, $C_3$ and $C_4$. Then  $28, 31 \in C_4$ and $(31, 28) \in E_5 $ being $(7,4) \in E_3$ (see~\eqref{pref} for the definition of $E_3$). Furthermore, $(1, 23) \in E_5$ because the vertices $1, 23 $ belong to different classes, namely $C_1$ and $C_3$, and $1+23$ is even. Note that, for $n \geq 4$, by construction, the induced directed graph $G_n \big [\, [ 2^{n-1}] \, \big ] $ coincides with the directed graph~$G_{n-1}$; furthermore, it is isomorphic to $G_n \big [\, [2^n ] \setminus  [ 2^{n-1}] \, \big ]$. 

Now, for $n \geq 4$ and $j \in [2^n]$,  we define a bracket
${\boldsymbol{\pi}^{(n; \, j)}} \in\Pi (n )$ such that 
\[
w_n (G_n ,{\boldsymbol{\pi}^{(n; \,  j)}}   ) =j\,. 
\] 
Thus, we find that any candidate $j\in [2^n]$ can win.

We will start from the permutations $\boldsymbol{\pi}^{(1)}, \ldots, \boldsymbol{\pi}^{(8)} $ introduced in formula \eqref{base}. For $i, j \in [8]$, let $\boldsymbol{\pi}^{(i)}(j)$ be the element at the $j$-th position of $\boldsymbol{\pi}^{(i)} $. For example, $\boldsymbol{\pi}^{(7)} (2)=8$, see~\eqref{base}.
For $j \in [2^n]$, we also define 
\[
b_j := \left \lceil \frac{j}{8} \right \rceil -1
\]
which means that $j \in C_{b_j +1}$. 

Consider an integer $n \geq 4$ and a number of candidates equal to $2^n$. For 
$j \in [2^n]$,  if $j$ is even, then we define the bracket ${\boldsymbol{\pi}^{( n; \, j)}} =
({\boldsymbol{\pi}^{( n; \, j)}}(1) , {\boldsymbol{\pi}^{( n; \, j)}} (2), \ldots , {\boldsymbol{\pi}^{( n; \, j)}} (2^n))$ as follows: 
\begin{equation}\label{ordinerem1}
{\boldsymbol{\pi}^{(n; \, j)}}(k ) := \left \{
 \begin{array}{ll}
 \boldsymbol{\pi}^{(j - 8 b_k )} ( k-8b_k )+8 b_k   &\text{, for } k \sim j;\\
 \boldsymbol{\pi}^{(1)}(k- 8 b_k ) + 8b_k     &\text{, for } k \not \sim j \text { with } k < j; \\ 
  \boldsymbol{\pi}^{(2)}(k -  8b_k) +  8b_k     &\text{, for } k \not \sim j  \text { with } k > j. 
\end{array} \right .
\end{equation}
Note that n the first line of \eqref{ordinerem1}, $(j - 8 b_k)$ and $(j -8 b_j)$ are equal, since $k \sim j$; moreover, 
$  (j - 8 b_k)  , \,    (k - 8 b_k) \in [8]$. Therefore,
$ \boldsymbol{\pi}^{(j - 8 b_k)} $ is one of the permutations defined in~\eqref{base}. See Example~\ref{EXAMPLE1} at the end of this section. 

For $j \in [2^n]$, if $j$ is odd, then we define the bracket ${\boldsymbol{\pi}^{(n; \, j)}}$ as follows:
\begin{equation}\label{ordinerem2}
\boldsymbol{\pi}^{(n; \, j)}(k) := \left \{
 \begin{array}{ll}
 \boldsymbol{\pi}^{(j - 8 b_k)} ( k-8b_k )+8 b_k   &\text{, for } k \sim j;\\
 \boldsymbol{\pi}^{(2)}(k- 8 b_k) + 8b_k     &\text{, for } k \not \sim j \text { with } k < j; \\ 
  \boldsymbol{\pi}^{(1)}(k -  8b_k) +  8b_k     &\text{, for } k \not \sim j  \text { with } k > j. 
\end{array} \right .
\end{equation}
Note that the permutations defined in~\eqref{ordinerem1} and~\eqref{ordinerem2} can be seen as compositions of permutations that individually act on the single classes $(C_q: q \in 2^{n-3})$. 

In the following theorem, we show that the directed graph $G_n$, seen as a majority graph, solves our
problem showing that any candidate can be the winner of the tournament simply by choosing a suitable bracket for the candidates.

\begin{theorem}\label{albero}
Let $n \geq 3$  be an integer, then for any $j \in [2^n]$ it holds that
\begin{equation} \label{formula-w}
w_n  (G_{n}, \boldsymbol{\pi}^{(n; \, j)})   = j\, . 
\end{equation}
\end{theorem}
\begin{proof} 
Lemma~\ref{8conc} expresses that~\eqref{formula-w} holds for $n=3$. For $n \geq 4$, let $ j \in [2^n]$ be an even number. Then candidate~$j $ wins the first three rounds because until the third round, the structure of the majority graph is the same in all classes, i.e., the directed graphs $G_n [C_q]$, for $q \in [2^{n-3}]$, are isomorphic. Hence, the first line of~\eqref{ordinerem1} establishes this result. Moreover,  all candidates of the form $8 k+1$, with $k \in \N_0 $ and such that $8k +1 < j$, win the first three rounds; see the second line of~\eqref{ordinerem1}. Finally, all players of the form $8k+2 $, with $k \in \N_0 $ and such that $j < 8k +2 \leq 2^n$, win the first three rounds; see the third line of~\eqref{ordinerem1}. By construction of $G_{n}$, the player $j$ will win, after the third round, matching any candidate $k$, with $k$ odd and $k<j$, and will win against any candidate $k$, with $k$ even and $k>j$. Therefore, candidate $j$ will be the winner of the tournament. 

For $n \geq 4$ and $j \in [2^n]$ odd, we use the construction in~\eqref{ordinerem2} for the bracket 
$\boldsymbol{\pi}^{(n; \, j)}$, but the proof is analogous to the previous case, and we leave it to the reader.
\end{proof}

The result obtained in Theorem~\ref{albero} for $n \geq 3 $ fails for $n =1,2$. 
For $n =1$, there are two candidates and a single match, so only one candidate can be the winner. For $n =2$, i.e., four candidates, any candidate, in order to win the tournament should win in two matches.
There are six possible matches. Hence, for any $G \in \mathcal{H} (2)$, only three of the candidates can win two matches and there is at least one candidate who can win at most one match, so this candidate cannot be the winner of the tournament for any bracket. 

\medskip 
Now, we give the definition of the seeded setting, which corresponds to the fact that the bracket has some constrains such that, for example, seeds 1 and 2 cannot match each other before the final (see, e.g., \cite{MS2023}).

\begin{definition} \label{seeded}
Given the set $[2^n]$ of candidates and a positive integer $k < n$, we say that an ordered sequence $S_{n,1},\ldots, S_{n,k}  \subset [2^n]$  is a  \emph{$k$-seeded system for the candidates $[2^n]$} if the following conditions holds:
\begin{itemize}
    \item[(a)] $S_{n,1},\ldots, S_{n,k}$ are disjoint subsets of $[2^n]$; 
    \item[(b)] $|S_{n,1}| =2 $ and $|S_{n, \ell} | = 2^{\ell -1}$, for $\ell = 2, \ldots , k$.
\end{itemize}
Furthermore, we say that a \emph{knockout tournament is 
$k$-seeded} if the candidates belonging to $S_1$ can match only in the final round and, for $u =2, \ldots, k$,  the candidates in $\bigcup_{\ell=1}^u S_\ell $ cannot match   before the $u$-th final round. 
A bracket for a $k$-seeded knockout tournament with $2^n$ candidates is said to be
\emph{valid} if it satisfies the seed constraints on the sets $S_{n,1},\ldots , S_{n,k}$.
\end{definition}

Without loss of generality, we choose once and for all the sets $S_{n,1},\ldots , S_{n,k}$. 
Define 
\begin{equation}\label{primo}
S_{n,1} = \{1,  2^{n-1} +1\}
\end{equation}
and, for $u \geq 2$, 
\begin{equation}\label{ric}
S_{n,u} =  \bigcup_{\ell =1}^{u-1}  (S_{n, \ell} +2^{n-u})\,, 
\end{equation}
where $S+a$, with $S$ a set of real numbers and $a$ a number, is defined as 
\[
S +a  = \{s+a: s\in S \}\,. 
\]

A different choice of $S_{n,1},\ldots , S_{n,k}$ corresponds to a renaming of the candidates; for this reason, all choices are equivalent. We will show that the brackets defined in \eqref{ordinerem1} and~\eqref{ordinerem2} are valid, i.e., they satisfy the constraints of the $k$-seeded knockout tournament with seeding candidates $S_{n,1},\ldots , S_{n,k}$ defined in \eqref{primo}.

\begin{theorem}\label{cor-seed1}
 For any $n \in \N $, any positive integer $k \leq n-3 $ and $i \in [2^n]$, 
the brackets $\boldsymbol{\pi}^{(n; \, i)}$ defined in \eqref{ordinerem1} and~\eqref{ordinerem2} are valid for a  
  $k$-seeded knockout tournament with $2^n$ candidates 
 and seeding candidates $S_{n,1},\ldots , S_{n,k}$ defined in \eqref{primo}. 
\end{theorem}

\begin{proof}
The bracket in which all candidates are ordered by increasing numbers, i.e. $(1,2,3, \ldots, 2^n)$, is valid. 
Now, any permutation $ \boldsymbol{\pi}^{(n; \, i)}$
can be seen as composition of permutations that individually act on a single class $C_q$. In fact, for $\ell \in [2^n]$, 
\[
\left \lceil \frac{\ell}{8} \right \rceil = \left  \lceil \frac{\boldsymbol{\pi}^{(n; \, i)}(\ell )}{8} \right \rceil.
\]
Therefore, after the first three turns, the seeding constraints satisfied by $(1,2,3, \ldots, 2^n)$ are also satisfied by $\boldsymbol{\pi}^{(n; \, i)}$. Thus, with $k \leq n-3 $, one finds that any $\boldsymbol{\pi}^{(n; \, i)}$ defined in~\eqref{ordinerem1} and~\eqref{ordinerem2} is valid, i.e., they satisfy the constraints of the $k$-seeded knockout tournament with seeding candidates $S_{n,1},\ldots , S_{n,k}$.  
\end{proof}

 Considering the majority graph $G_n$ and applying Theorems~\ref{albero}--\ref{cor-seed1}, we observe that for  $k \leq n-3$, every candidate can win a $k$-seeded tournament provided a suitable bracket is chosen.
We conclude this section by presenting an example for a $1$-seeded system. 
 
\begin{example} \label{EXAMPLE1}
{\rm
 Consider 16 candidates, with candidates 1 and 9 only meeting in the last round, i.e., $S_{4,1} =\{1, 9\}$. Considering the majority graph $G_4= ([16], E_4)$ and choosing the bracket $\boldsymbol{\pi}^{(4; \, i)}$, we can make any candidate $i \in [16]$ win the 1-seeded tournament. For instance, if we want candidate 12 to win the tournament, then the bracket in~\eqref{ordinerem1} is
 \arraycolsep=2pt
 \begin{eqnarray}\label{exampleGB}
    \boldsymbol{\pi}^{(4; \, 12)}&=&(1, 2, 3, 4, 5, 6, 7, 8,     4 +8, 1+8, 8+8, 2+8, 5+8, 6+8, 7+8, 3+8 ) \nonumber  \\
&=&(1, 2, 3, 4, 5, 6, 7, 8,     12, 9, 16, 10, 13, 14, 15, 11 )\,. 
        \end{eqnarray}
The first eight components of this bracket correspond to $\boldsymbol{\pi}^{(1)}$, as prescribed by the second line of~\eqref{ordinerem1}. The last eight components correspond to $\boldsymbol{\pi}^{(4)}$, with each component translated by 8, as prescribed by the first line of~\eqref{ordinerem1}. The third line of~\eqref{ordinerem1} is not used in this case because there are only two classes $C_1$ and $C_2$. Candidates~1 and 12 will advance to the final round and candidate~12 will beat candidate~1 because (12,1) belongs to $E_4$ (since $12+1$ is odd, as explained below~\eqref{GI}). Furthermore, the constraints imposed by $S_{4,1} =\{1, 9\}$ are satisfied, because 1 is in the first half of the bracket and 9 in the second half. Therefore, the bracket in~\eqref{exampleGB} is valid according to Definition~\ref{seeded}. 

We emphasize that we have explicitly constructed both the majority graph and the bracket, which turned out to be particularly simple. 
}
\end{example}
 
\section{Construction of the voting profile} \label{construction voting}
In this section, we construct, in an algorithmic way, a voting profile that generates the majority graph $G_{n}$, for any integer $n \geq 3$. McGarvey~\cite{McGarvey1953} proved that any majority graph can be generated by a suitable voting profile. Several years later, Stearns~\cite{Stearns} proved the same result while reducing the number of voters needed. In particular, he showed that $m+2$ (resp.~$m+1$) voters suffice to obtain any given majority graph on a set of $m$ candidates when $m$ is even
(resp.~odd). Erd{\H{o}}s and Moser~\cite{ErdosMoser} designed an algorithm that generates any majority graph on a set of $m$ candidates with $c_1 m /\log{m}$ voters, where $c_1$ is a fixed positive constant. Therefore, the result of 
Erd\H{o}s and Moser, at least for large $m$, requires a smaller number of voters than that of Stearns. Our problem is related to these constructions, but we are only interested in a specific voting profile that generates $G_{n}$ when there are $2^{n}$ candidates. 
To that end, we will use the solution of Stearns~\cite{Stearns}, Erd\H{o}s--Moser's method and Lemma~\ref{EMnew} below,
which turns out to establish a link between~\cite{ErdosMoser} 
and~\cite{Stearns}.

Next, we will use the following simple notation.

\begin{definition} \label{card-commission}
The \emph{size} of a voting profile $R$ is denoted by $\# R$; it is the number of rows in $R$, i.e., the number of preference lists that form $R$. 
\end{definition}

\begin{lemma}\label{EMnew} 
Let $c>1$ and $r>0 $ be integers. Suppose that a family of finite directed graphs 
 $(G^{(\ell )} =( V^{(\ell )} , { E}^{(\ell)})  : \ell \in [c ])$ is such that $V^{(i )} \cap V^{(j )} = \emptyset$, for any $i\neq j \in [c]$. 
Assume that 
$(R^{(\ell )}:\ell \in [c ])$ 
is a family of voting profiles
such that $R^{(\ell)}$ generates $G^{(\ell)}$ and $\# R^{(\ell)}  = 2r $, for $\ell \in [c ]$. 
Then there exists a voting profile $R$ with $\# R =2r $ that generates $G:= (\bigcup_{\ell=1}^c V^{(\ell)}, \bigcup_{\ell =1}^c {E}^{(\ell)} )$. 
\end{lemma}

\begin{proof}
We directly provide the voting profile $R$ that generates $G$ with $\# R = 2r$. For odd $s \in [2r] $, the $s$-th row of $R$ is defined as the concatenated vector
\[
R_{s, \cdot}  =
\left( R_{s, \cdot}^{(1)} , R_{s, \cdot}^{(2)} , 
\ldots, R_{s, \cdot}^{(c)}  \right), 
\]
whereas for even $s\in [2r]$, the $s$-th row of $R$ is defined as
\[
R_{s, \cdot}  =
\left(R_{s, \cdot}^{(c)}, R_{s, \cdot}^{(c-1)} , 
\ldots ,  R_{s, \cdot}^{(1)}  \right).
\]
Let $i, j \in V^{(\ell)} $ for $\ell \in [c]$ and $s \in [2r]$. Then $i$ precedes $j$ in $R_{s, \cdot}$  if and only if $i$ precedes $j$ in $R_{s, \cdot}^{(\ell)}$. Therefore,  
 $R_{s , \cdot}$ and $R_{s,\cdot}^{(\ell)}$ give rise to the same number of votes for $i$ and therefore also for $j$.
Thus, in particular, $R_{s,\cdot} $ and $R_{s , \cdot}^{(\ell)}$ either generate the same arrow between $i$ and $j$ or no arrow at all.

Next, let $i \in V^{( \ell_1 )} $ and $j \in V^{( \ell_2 )}$ for $\ell_1,\ell_2 \in [c]$. Consider without loss of generality $\ell_1 <\ell_2$. For any odd $s \in [2r]$,  $i$  precedes $j$ in $R_{s,\cdot}$, whereas for any even $s \in [2r]$, $j$  precedes $i$ in $R_{s,\cdot}$. Hence, $i$ precedes $j$ and $j$ precedes $i$
exactly $r$ times. Therefore, in this case, there is no arrow between $i $ and $j$ as intended. In fact, by hypothesis, the directed graphs $G^{(\ell_1)} $ and $G^{(\ell_2)} $ are not connected. 
\end{proof}

The following remark considers the same framework as in the previous lemma. 

\begin{rem}\label{stessi-voti}{\rm
Consider the voting profile $R$ constructed in the previous proof. Let $i, j \in  V^{(\ell )}$ for some $\ell \in [c]$. We have already noted in the proof of Lemma~\ref{EMnew} that 
\[
i \text{ precedes } j \text{ in } R_{s , \cdot} 
\iff i \text{ precedes } j \text{ in } R^{(\ell )}_{s , \cdot} \,.
\]
Hence, the number of preference lists of $R$ in which $i$ precedes $j$ is equal to the number of preference lists of $R^{(\ell)}$ in which $i$ precedes $j$. From the point of view of the match between $i$ and $j$, this means that $i$ and $j$ obtain the same number of votes in the case where the candidate set is $\bigcup_{r=1}^c V^{(r)} $ and the voting profile is $R$ as in the case where the candidate set is $V^{(\ell)}$ and the voting profile is $R^{(\ell)}$.}
\end{rem}

For the candidate set~$\{1,2, \ldots, 8 \}$, Stearns' algorithm~\cite{Stearns} allows to determine a voting profile that generates $G_{3} = ([8] , E_3)$.
By direct analysis of the 28 pairs $\{i,j\}$ with $i \neq j \in [8]$ one can easily check that $R_3$ generates $G_{3}$. Moreover, in every match between two candidates belonging to $[8]$, the winner gets six votes and the loser gets four votes. 
\bigskip 
\begin{equation} \label{St}
R_3= 
\begin{bmatrix}
  8 & 1 &5 & 2 & 3 & 4 & 6 & 7 \\
 2 & 7 & 6 & 4 & 1 & 3 & 5 & 8 \\
 2  & 7 & 3 & 6 & 4 & 5 & 8 & 1 \\
 4 & 1 & 3 & 8 & 5 & 6 & 7 & 2 \\
  3 & 5 & 2 & 6 & 7 & 1 & 4 & 8 \\
  6 &  8 & 4 & 1 & 5 & 7 & 2 & 3  \\
 6 & 7 & 3 & 4 & 8 & 1 & 2 & 5 \\ 
  8 & 5 & 2 & 7 & 1 & 4 & 3 & 6\\ 
 7 & 8 & 5 & 6 & 3 & 4 & 1 & 2 \\
  1 & 2 & 3 & 4 & 5 & 6 & 7 & 8
\end{bmatrix}
\end{equation}

\medskip 
Now, let us introduce a notation for the construction carried out in the proof of the previous lemma for two graphs only. 

\begin{definition}\label{bowtie}
Let $G^{(1) } =( V^{(1)} , { E}^{(1)})  $ and 
 $G^{(2) } =( V^{(2)},{ E}^{(2)})$ be two directed graphs 
such that $V^{(1)}  \cap  V^{(2)} = \emptyset $. Let 
$R^{(1)}$ (resp.\ $R^{(2)}$) be a voting profile that  generates  $G^{(1)} $  (resp.\ $G^{(2)}$) 
and such that 
 $\# R^{(1 )} = \# R^{(2)}=2r$. The voting profile 
$R^{(1)} \bowtie R^{(2)}$ on the set of candidates (or vertices) $  V^{(1)} \cup V^{(2)}$ is such that  
\begin{itemize}
\item[(1)] for odd $s \in [2r]$, the $s$-th row is the concatenated vector 
\[
 ( R^{(1 ) } \bowtie R^{(2)})_{s, \cdot} 
:=
 (R^{(1 )}_{s, \cdot}, 
 R^{(2 )}_{s , \cdot})\,;
\]
\item[(2)]  for even $s\in [2r] $, the $s$-th row is
\[
( R^{(1)} \bowtie R^{(2)})_{s, \cdot} 
:=
 (R^{(2 )}_{s,\cdot}, 
 R^{(1 )}_{s,\cdot})\, . 
\]
\end{itemize}
In particular, $\# ( R^{(1)} \bowtie R^{(2)}   ) =\# R^{(1)}  =\# R^{(2)} = 2r$. 
\end{definition}

We also define the matrix $\boldsymbol{1}$ formed of all ones. When we write $R  + 2^n \cdot \boldsymbol{1}$, we mean that each element of $R$ is increased by $2^n$.
\medskip 

Now we define four vectors associated with the set of candidates $ [2^{n}]$, for $n \geq 4$. Let 
\begin{equation}\label{abab}
\begin{array}{ll}
A^{(n)}_1 =(u  \in [2^{n-1}]: u \text  { is odd}  )= (1,3, \ldots , 2^{n-1} -1)\,, \\
B^{(n)}_1 =( u  \in [2^{n}]\setminus [2^{n-1}]   :  u \text  { is odd})  
= (2^{n-1}+1 ,  2^{n-1}+3, \ldots , 2^{n}-1 ) \,,   \\
A^{(n)}_2 = ( u  \in [2^{n-1}]: u \text  { is even}  )=(2,4, \ldots , 2^{n-1})\,, \\
B^{(n)}_2 =( u  \in [2^{n}]\setminus [2^{n-1}]   :  u \text  { is even}) = 
(2^{n-1}+2 ,  2^{n+2}+4, \ldots , 2^{n})\,, 
\end{array}
\end{equation}
where the elements of the vectors are in increasing order. 

Now, let us define the preference lists on the set of candidates $[2^{n}]$ as 
\begin{equation} \label{iterasimile}
\begin{array}{c}
v^{(n)}_1 =(A^{(n)}_1 ,  B^{(n)}_1 , A^{(n)}_2 ,  B^{(n)}_2) \,, \\
v^{(n)}_2 =( \bar A^{(n)}_2 ,  \bar B^{(n)}_2 ,  \bar A^{(n)}_1   ,    \bar B^{(n)}_1)  \,,  \\
v^{(n)}_3 =  (B^{(n)}_1 , A^{(n)}_2 ,  B^{(n)}_2,  A^{(n)}_1 )  \,, \\
v^{(n)}_4 =( \bar B^{(n)}_2 , \bar  A^{(n)}_1 ,  \bar B^{(n)}_1 , \bar A^{(n)}_2 ) \,.
\end{array}
\end{equation}
For an integer $n \geq 4$,  we recursively define
\[ 
R_{n} =\left  ( v^{(n)}_1 ,   v^{(n)}_2 ,   v^{(n)}_3 ,   v^{(n)}_4 ,            \left[   R_{n   -1} \bowtie 
 ( R_{n   -1}  + 2^{n+2} \cdot \boldsymbol{1}  ) \right ]  \right )_V\, .
\] 

In particular, this construction shows that $\# R_{n+1} =  \# R_{n} + 4$, therefore $\# R_{n} =4(n-3) +\# R_3 =4 n -2$.  We also consider the voting profile ${\tilde R}_{n}$ which is identical to $R_{n} $ with the last row being deleted. Thus, $\#{ \tilde R}_{n} = 4 n -3$.

\begin{theorem} \label{tf}
For any integer $n \geq 3$, both voting profiles $R_n $ and ${ \tilde R}_{n} $  generate $G_{n} = ([2^{n}] , E_n )$.  Moreover, if $(i,j)  \in  E_n $ then $i $ precedes $j $ in  $2n $ preference lists of $R_n $,
whereas $j$ precedes $i$ in $2n-2$ preference lists of $R_n$.
\end{theorem}

\begin{proof}
The proof proceeds by induction. For $n =3$, we have already shown that voting profile $R_3$ generates $G_{3}$ (see~\eqref{St}). 

Consider $G_{n}= ([2^{n}],  E_n)$ and $G_{n+1}= ([2^{n+1}], E_{n+1})$.
Now, we show that  
\[
R_n \text{ generates } G_n  \implies R_{n+1}  \text{ generates } G_{n+1}\,. 
\]
First, we observe that by applying Lemma~\ref{EMnew}, see also Lemma~2 of \cite{ErdosMoser}, we obtain that $ v^{(n+1)}_1$ and $v^{(n+1)}_2$ generate
all and only the arrows of $E_{n+1}$ pointing from the vertices of~$[2^{n}]$ to the vertices of $[2^{n+1}] \setminus [2^{n}]$, that is, there is an arrow from $i \in [2^{n}] $ to $j \in [2^{n+1}] \setminus [ 2^{n} ]$ if and only if $i+j $ is even. Similarly, we obtain that $ v^{(n+1)}_3$ and 
$v^{(n+1)}_4$ generate 
all and only the arrows of $E_{n+1}$  pointing from the  vertices in  
$ [2^{n+1}]  \setminus [ 2^{n}]$  to the vertices in  $[2^{n}]$. 

If we eliminate all the arrows already generated by the first four preference lists, then we get two disconnected directed graphs that are isomorphic to $G_{n}$. The first graph coincides with $G_{n}$ and the second, denoted 
\[
G'_{n}= ( [2^{n+1}] \setminus  [2^{n}], {E}^{\prime}_{n})= G_{n+1} \big [ [2^{n+1}] \setminus  [2^{n}] \big ]\,, 
\]
is obtained from $G_{n}$ by mapping any vertex $i \in [2^{n}]$ to $i+2^{n}$.

\medskip 
Now, by Lemma~\ref{EMnew}, we obtain that 
 $R_{n} \bowtie ( R_{n}  + 2^{n+3} \cdot \boldsymbol{1}) $ generates all the arrows 
 of the two isomorphic directed graphs $G_{n}$ and $G'_{n}$.

Now, for $n \in \N$ and under the hypothesis that $(i,j) \in E_{n}$, we show that $i$ precedes $j $ in $2n $ preference lists of $R_n$ and $j $ precedes $i$  in the remaining $ 2n -2 $ ones. Let us analyze  the three possible cases. 
\begin{itemize}
\item[(a)] The case $i \in [2^{n-1}]$ and $j \in [2^{n}]   \setminus  [2^{n-1}]$  with $i+j $ even. Then $i $ precedes $j$ in $v^{(n)}_1$ and in $v^{(n)}_2$. The other preference lists of $R_n$ are equally divided between those in which $i$ precedes $j$ and those in which, vice versa, $j$ precedes $i$. 
\item[(b)] The case $i \in [2^{n}]   \setminus  [2^{n-1}]$ and $i \in [2^{n-1}]$ with $i+j $ odd. Then $i$ precedes $j$ in $v^{(n)}_3$ and in $v^{(n)}_4$.The other preference lists of $R_n $ are equally divided between those in which $i$ precedes $j$ and those in which, vice versa, $j$ precedes $i$. 
\item[(c)] The cases where both $i, j \in [2^{n-1}]$ or both $i, j \in [2^{n}]   \setminus  [2^{n-1}]$. 
Then, among
$v^{(n)}_1$, $v^{(n)}_2$, $v^{(n)}_3$, and $v^{(n)}_4$, there are two preference lists in which $i$ precedes $j $ and two preference lists in which $j$ precedes $i$. 
Clearly, $\#  \left[   R_{n   -1} \bowtie  ( R_{n-1}  + 2^{n+2} \cdot \boldsymbol{1}  ) \right ] = 4n -6$. 
By Remark~\ref{stessi-voti} and induction, it holds that
$ R_{n-1} \bowtie  ( R_{n-1}  + 2^{n+2} \cdot \boldsymbol{1})$ contains $2(n-1)$ preference lists in which $i $ precedes $j $, whereas $j$ precedes
$i$ in the remaining $2(n-1)-2$ preference lists. Therefore, 
the total number of preference lists in which $i$ precedes $j$ is $2n$ in  $R_n$. 
\end{itemize}

 Now, we prove that $\tilde {R}_n$  generates $G_{n}$. We have seen that for any $i , j \in [2^{n}]$
 there is a difference of two preference lists in $R_n $ in favor of the winner of the match. Therefore, $R_n$ and $\tilde {R}_n$ give rise to the same result for any match.
\end{proof}
\begin{example}
{\rm
We construct the voting profile $R_4$ that generates $G_{4} =([2^4], E_4)$ starting from $R_3$. We recall that 
\[ 
R_{4} =\left  ( v^{(4)}_1 ,   v^{(4)}_2 ,   v^{(4)}_3 ,   
v^{(4)}_4 ,            \left[   R_{3} \bowtie 
 ( R_{3}  + 2^{3} \cdot \boldsymbol{1}  ) \right ] \right )_V .
\]
 By applying Lemma~2 of \cite{ErdosMoser} 
 to 
\[
A_1 = ( 1,3,5,7), B_1 =( 9,11,13,15    )       , A_2 = ( 2,4,6,8),  B_2 =( 10,12,14,16  ) \, ,
\]
we obtain, associated with the first two voters, the following  preference lists  
\begin{equation} \label{iterata}
\arraycolsep=2pt
\begin{array}{rl}
 v^{(4)}_1& =( 1,3,5,7,9,11,13,15,2,4,6 ,8,10,12,14,16) \\
 v^{(4)}_2&= (8,6,4,2,16,14,12,10, 7,5,3,1, 15,13,11,9 )\,,
\end{array}
\end{equation}
and, similarly, for two other voters
\begin{equation} \label{iteratasec}
\arraycolsep=2pt
\begin{array}{rl}
 v^{(4)}_3 &=( 9,11,13,15,2,4,6 ,8,10,12,14,16, 1,3,5,7) \\
 v^{(4)}_4&= (16,14,12,10, 7,5,3,1, 15,13,11,9, 8,6,4,2 )\,.
\end{array}
\end{equation}
Now,  by eliminating from $G_{4}$ all the arrows already generated with the first four voters, we obtain two graphs that are isomorphic to $G_{3}$. Hence, by applying Lemma~\ref{EMnew}, 
we can construct the remaining preference lists of ten voters to complete the voting profile. 
Hence, voting profile $R_4$ is composed of the following preference lists: 
\[
\arraycolsep=2pt
\begin{array}{rl}
 v^{(4)}_1 &=( 1,3,5,7,9,11,13,15,2,4,6 ,8,10,12,14,16) \\
 v^{(4)}_2&= (8,6,4,2,16,14,12,10, 7,5,3,1, 15,13,11,9  ) \\
 v^{(4)}_3 &=( 9,11,13,15,2,4,6 ,8,10,12,14,16, 1,3,5,7) \\
 v^{(4)}_4&= (16,14,12,10, 7,5,3,1, 15,13,11,9, 8,6,4,2 ) \\
  v^{(4)}_5&= (8 , 1 ,5 , 2 , 3 , 4 , 6 , 7, 16, 9, 13, 10, 11, 12,14,15 ) \\
  v^{(4)}_6&= (10, 15, 14, 12, 9, 11, 13, 16 , 2 , 7 , 6 , 4 , 1 , 3 , 5 , 8)\\ 
  v^{(4)}_7 &= (2  , 7 , 3 , 6 , 4 , 5 , 8 , 1 ,  10, 15, 11, 14, 12, 13,16,9)   \\
 v^{(4)}_8 &= (12, 9, 11,16,13,14, 15, 10,   4 , 1 , 3 , 8 , 5 , 6 , 7 , 2) \\
 v^{(4)}_9 &= ( 3 , 5 , 2 , 6 , 7,  1 , 4 , 8, 11, 13, 10, 14, 15, 9, 12, 16)   \\
 v^{(4)}_{10}&= (14, 16,  12, 9, 13, 15, 10, 11,   6 ,  8 , 4 , 1 , 5 , 7, 2 , 3)\\
 v^{(4)}_{11}&= ( 6 , 7 , 3 , 4 , 8 , 1 , 2 , 5 , 14,  15, 11, 12, 16, 9,  10, 13) \\
 v^{(4)}_{12}&= ( 16, 13 , 10, 15, 9, 12, 11, 14, 8 , 5 , 2 , 7 , 1 , 4 , 3 , 6 ) \\
 v^{(4)}_{13}&=( 7 , 8 , 5 , 6 , 3 , 4 , 1 , 2 ,  15, 16, 13, 14,11, 12, 9, 10 ) \\ 
 v^{(4)}_{14}&= (9, 10,  11, 12,13, 14,  15, 16, 1 , 2 , 3, 4 ,  5 , 6 , 7 , 8)
\end{array}
\]
}
\end{example}

We present a simple corollary of Theorem \ref{tf} showing that for a knockout tournament with $2^n$ candidates, for any $m\geq 4n-3$ there exist voting profiles of size $m$ such that all candidates are admissible winners of the tournament for suitable brackets. 

\begin{corollary}\label{corSIZE}
    For any integer $n \geq 3$ and any $m \geq 4n-3$, there exists a voting profile~$R$ that generates $G_{n}$ of size $\#R =m $.
\end{corollary}

\begin{proof}
    Let $G$ be a directed graph and $R$ a voting profile that generates $G$. Then, it is well known that one can construct a voting profile $R'$ that generates $G$ with $\#R' = \#R +2$. This can be done by adding a preference list $v'$ and the reverse preference list $\bar {v'}$, see~\cite{McGarvey1953}. Therefore, starting from $\tilde{R}_n$, one can construct a voting profile $R'$ of odd size $m \geq 4n-3$ generating $G_n$. Analogously, starting from $R_n$, one can construct a voting profile $R'$ of even size $m\geq 4n-2$  generating $G_n$.
\end{proof}

\section{Random Number of Voters}\label{sec:random}

In many practical situations, the number of voters is random,
such as the viewers of a television program or the electorate in a political election (see, e.g., \cite{FP99, makris2008, CS23}). In the context of a random number of voters, the Poisson framework has become the natural and standard assumption.
In the following,  we will make use of the well-known thinning property of the Poisson random variable (see, e.g., \cite{Bre2020}).

In the previous section, associated with a number of $2^{n} $ candidates, we have constructed the voting profile $R_{n}$ consisting of $4n-2$ preference lists, for $n \geq 3$. Now, we assume that the number of voters is a random variable following a Poisson distribution $\mathrm{Poi}(\lambda_n)$, and each voter selects a preference list independently and uniformly at random from $R_{n }$. By the thinning property, in this way, we obtain $4n-2$ independent Poisson distributions with the same expectation $\lambda_n/ (4n-2)$, say $(W^{(n)}_1 ,\ldots ,W^{(n)}_{4n -2})$, where $W^{(n)}_\ell$ is the number of voters who have selected the $\ell$-th row of $R_n$. We denote by 
${\hat R }_n (W^{(n)}_1 ,\ldots ,W^{(n)}_{4n -2} )$ the random voting profile having 
$W^{(n)}_\ell$ rows (preference lists) equal to the $\ell$-th row of $R_n$, for every $\ell \in [4n-2]$. The (random) majority graph generated by ${\hat R}_n (W^{(n)}_1 ,\ldots ,W^{(n)}_{4n -2} )$ is denoted by
\[
\widehat G_n (W^{(n)}_1 ,\ldots ,W^{(n)}_{4n-2}   )  =\left([2^{n}], \hat{E}(W^{(n)}_1 ,\ldots ,W^{(n)}_{4n -2}) \right)\,.
\]

We are interested in understanding when, with high probability, the majority graph~$\widehat G_n$, associated with the random preference profile, coincides with the majority graph $G_n$, defined in Section~\ref{sec2bis}. 
Therefore, for $n \geq 3 $, we will study the following events
\[
F_n^{(\lambda_n)} =\left \{   \widehat G_n   (W^{(n)}_1 ,\ldots ,W^{(n)}_{4n -2} )   \neq G_n \right \}, 
\]
where $\sum_{\ell =1}^{4n-2}W^{(n)}_\ell \sim  \mathrm{Poi}(\lambda_n) $. 

Now, we need a brief introduction to Chernoff bound. 
For a real random variable~$W$ and a real constant $a <\mathbb{E}(W)$, the Chernoff bound states that 
\[
\mathbb{P}(W \leq a ) \leq e^{- I_W (a)}\,,
\]
where $I_W$ is the  \emph{rate function of} $W$: 
\[
I_W (a): = \sup_{t \leq 0 } \{ at - \ln \left (\mathbb{E}(e^{tW}) \right )\}\,. 
\]
It corresponds to the  Legendre--Fenchel transform of the generating function.

Let $X$ and $Y$ be independent random variables with $ X \sim  \mathrm{Poi}(\lambda)$, $ Y \sim  \mathrm{Poi}(\mu)$ and $  \lambda > \mu >0 $. We consider the random variable $(X-Y)$ having $\mathbb{E}(X-Y) >0$.  The rate function $I_{X-Y}$ of the random variable~$X -Y$ calculated at zero is 
\[
I_{X-Y} (0) =\sup_{t \leq 0} \Big\{- \ln \big(\mathbb{E} (e^{t (X -Y)}) \big)\Big\}= \sup_{t \leq 0}  \Big \{-\ln \big(\mathbb{E} (e^{t X}) \mathbb{E} (e^{-t Y}) \big)      
\Big \}=
\]
\begin{equation}
= \sup_{t \leq 0}  \Big \{-\ln \big(\mathbb{E} (e^{t X}) \big ) -\ln \big (\mathbb{E} (e^{-t Y}) \big) \Big \}
 =\sup_{t \leq 0} \big\{ \lambda + \mu - \lambda e^t - \mu e^{-t} \big\}. 
\end{equation} 
The last equality follows from the fact that the
moment-generating function of $X $ is given by
$\mathbb{E} (e^{tX}) = e^{\lambda (e^t -1)}$ and that of $Y$ 
by
$\mathbb{E} (e^{-tY}) = e^{\mu (e^{-t} -1)}$. Differentiating, one obtains that the maximum of $\lambda + \mu - \lambda e^t - \mu e^{-t} $ is attained at $t = \frac{1}{2}\ln (\frac{\mu}{\lambda}) <0 $. Therefore, 
 \begin{equation}\label{ratef} 
 I_{X-Y} (0) =  \lambda + \mu -2 \sqrt{\lambda \mu}\,. 
\end{equation}

\medskip 

In the following theorem, we consider $2^{n}$ candidates with $n \geq 3 $ and a random number of voters following a Poisson distribution $\mathrm{Poi}(\lambda_n)$, 
with each voter randomly selecting a preference list independently and uniformly from the preference lists of $R_{n}$. 
If $\lambda_n$ is larger than $16 (\ln (2)+ \epsilon) n^3$, with $\epsilon >0 $, then $\mathbb{P} (F_n^{(\lambda_n)} )$
exponentially decreases to zero when $n$ grows to infinity. 
Hence, the following Theorem~\ref{random} will state that the result of the deterministic voting profile and the randomly generated one yield the same majority graph with high probability, provided $\lambda_n $ is large. 
Therefore, even in the random case, one could predict, with high probability, the outcome of the knockout tournament only by knowing the initial order of the candidates.  

We are ready to state the following theorem.

\begin{theorem}\label{random}
Consider the sequence of events $(F_n^{(\lambda_n)} )_{n \in \N_0}$. If 
\begin{equation}\label{liminf}
    \liminf_{n \to \infty}\frac{\lambda_n}{n^3} \,>\, 16\ln{2} \,,
\end{equation}
then there exist $K>0$ and $\beta>0 $ such that, for any $n \in \N$,
\[
\mathbb{P} (F_n^{(\lambda_n)} ) < K e^{-\beta n } \,  . 
\]
Therefore, $\lim_{n \to \infty } \mP(F_n^{(\lambda_n)})=0$.
\end{theorem}

\begin{proof}
  In the deterministic case, voting profile $R_n $ with $\# R_n = 4n -2$ generates the majority graph $G_{n}= ([2^{n}  ] ,  E_n )$.  
Let  $ i,j  \in [2^{n}]$. Without loss of generality, we suppose that $(i , j )$ is an arrow of $ E_n$. Hence, among the $4n -2$ preference lists, there are  $2n$ preference lists in which 
 $i$ precedes $j$ (see Theorem~\ref{tf}). 
Without loss of generality, we also suppose that the first $2n $ preference lists of $R_n$ are in favor of $i $, and that the remaining $2n -2 $  preference lists are in favor of $j$.  

Let $ X^{(n)} := \sum_{k =1}^{2n} W^{(n)}_k$ and $ Y^{(n)} : = \sum_{k =2n+1}^{4n-2} W^{(n)}_k $. 
Since the random variables $ X^{(n)}$ and 
$ Y^{(n)}$ are independent and distributed as 
$\displaystyle   \mathrm{Poi} \left (\frac{2n}{4n-2}  \lambda_n  \right  ) $  and    $\displaystyle \mathrm{Poi} \left  ( \frac{2n-2}{4n-2} \lambda_n  \right  ) $, respectively,
the rate function in~\eqref{ratef} becomes 
\arraycolsep=2pt
\begin{eqnarray} \label{newforexample}
I_{ X^{(n)} - Y^{(n)} } (0)& =& \lambda_n - 2 \sqrt{\left (\frac{\lambda_n}{2}+ \frac{\lambda_n}{4n-2}            \right   )       \left (\frac{\lambda_n}{2}- \frac{\lambda_n}{4n-2}            \right   )           } \nonumber   \\&=& \lambda_n \left ( 1-  \sqrt{1 -   \frac{1}{(2n-1)^2}                     }  \right )  \\&=& \lambda_n \left (\frac{1}{2(2n-1)^2 }+ O \left ( \frac{1}{(2n-1)^4 }\right ) \right ) \nonumber
>
c (2n-1)              \,, 
\end{eqnarray}
for $n $ large enough and some $c > \ln{2}$. The last equality follows by Taylor expansion, and the inequality follows from the hypothesis in \eqref{liminf}.

For the random voting profile one has $(i, j) \not \in  \hat{E} (W^{(n)}_1 ,\ldots ,W^{(n)}_{4n -2})$   
if and only if $X^{(n)} - Y^{(n)}   \leq 0$.  Hence, 
by the Chernoff bound, we have 
\arraycolsep=2pt
\begin{eqnarray*}
\mathbb{P}((i , j ) \not \in  \hat{E} (W^{(n)}_1 ,\ldots ,W^{(n)}_{4n -2}   )  )
&=&\mathbb{P}( X^{(n)}  \leq  Y^{(n)}  ) \\
&\leq & \exp{\left \{ -I_{ X^{(n)} - Y^{(n)}  }  (0)\right \} }
\leq \exp{\left \{ -  c (2n-1) \right \}}\,, 
\end{eqnarray*}
as before, for some $c > \ln{2}$  and $n$ large enough. 
The number of arrows  in $E_n $  is given by $2^{n-1}(2^{n}-1)< 2^{2n-1}$. Therefore, by the union bound, 
\arraycolsep=2pt
\begin{eqnarray*}\label{probe}
\mP (F_n^{(\lambda_n)} ) &=& \mathbb{P}\left ({\hat R }_n (W^{(n)}_1 ,\ldots ,W^{(n)}_{4n -2}   ) \text{ does not generate } G_{n} \right )\\
&& \leq  
 2^{2n-1} \exp{\left \{ -  c (2n-1)               \right \} }= \exp{\left \{(\ln{2}-c) (2n-1)\right\}  =e^{c- \ln 2} \cdot e^{2(\ln{2} - c)n}          } \,. 
\end{eqnarray*}
Hence, taking $\beta = 2 ( c- \ln{2}) >0 $ and  $ K= e^{\beta/2}$ one has 
\[
\mP (F_n^{(\lambda_n)} ) <  K e^{- \beta n          } , 
\]
for each $n\geq 3 $. Therefore, $\lim_{n \to \infty } \mP (F_n^{(\lambda_n)} )  =0$.
\end{proof}

\begin{rem} {\rm
We highlight that similar results could be obtained for other distributions through the Chernoff bounds, for instance, for the multinomial distribution. In the latter case, the number of voters should be fixed, and the distribution over the preference lists is multinomial with identical class probabilities, again corresponding to uniform selection from the deterministic voting profile $R_n$.
}
\end{rem}

\begin{example}{\rm
    In this example, we determine the values of $\lambda_3$ that guarantee $\mathbb{P}(F_3^{(\lambda_3)} ) \leq\frac{1}{100} $. 
    Hence, we are considering $2^3 = 8$ candidates.
    By \eqref{newforexample}, for $n =3 $, we obtain
    \[
    I_{X^{(3)}-Y^{(3)}}(0) = \lambda_3 \left(1 -\frac{1}{5}\sqrt{24} \right)\,. 
    \]
    By the Chernoff bound and the union bound over the pairs of candidates, we have
       \[
       \mathbb{P} (   F_3^{(\lambda_3)}   ) \leq 
     \binom{8}{2} \cdot e^{-I_{X^{(3)}-Y^{(3)}}(0)}=28 \cdot  e^{-\lambda_3 \left(1 -\frac{1}{5}\sqrt{24} \right)}\,.
       \]
     Hence, $\lambda_3 \geq 393.86$ implies $ \mathbb{P} (F_3^{(\lambda_3)}  ) \leq \frac{1}{100}$. Therefore, 
     one might consider the following situation. A random number of people, Poisson-distributed with $\lambda=394$, randomly draw with replacement preference lists from an urn containing the ten lists of $R_3$. At the end of this random procedure, the tournament organizer receives a fraudulent request to have candidate $i\in [8]$ win the tournament. The organizer, without actually knowing the voting profile, can then use the bracket $\boldsymbol{\pi}^{(i)}$ and the chosen candidate $i$ will have a probability of winning greater than $99\%$.}
\end{example}

\section{Conclusion} \label{conc}
In this paper, we have shown that for any knockout tournament with $2^n$ candidates ($n \geq 3$), there exist voting profiles
that admit any candidate as the winner of the tournament. This can be accomplished by adjusting the bracket. An interesting generalization is the problem of whether any two candidates could appear as finalists.

We emphasize that the majority graphs $(G_n )_{n \geq 3}$, the voting profiles $({\tilde R}_n)_{n \geq 3}$ and the brackets necessary to obtain the result, can be explicitly constructed, see Theorems~\ref{albero}--\ref{tf}. Moreover, all of these constructions require a few steps only and can easily be implemented on a computer. Note that the voting profile ${\tilde R}_{n}$, acting on the set of candidates $[2^n]$, contains $4n -3$ preference lists, which means that the size of the commission increases logarithmically with the number of candidates present in the knockout tournament. 

We have also tackled the same problem in the setting of a Poisson-distributed number of voters. Using Chernoff bounds, we have shown that the results of the deterministic and random voting profiles are, with high probability, identical when the expected number of voters is larger than $C n^3$, with $C>16 \ln 2$, see Theorem~\ref{random}. 

For a larger number of candidates, it seems infeasible that a single voter could order all of them in a preference list. Therefore, the case where each voter only presents her top-$k$ candidates seems like an interesting path to explore in future work.

Finally, we mention that one can also associate a majority graph with a set of random variables (see, e.g.,~\cite{MontesEtAl, SaSa2024,steinhausparadox}). Recently, it was shown that the restricted and particular class 
of hitting times for Markov chains is able to generate any majority graph~\cite{D21ranking}. Furthermore, in the scientific literature, the concept of majority graph has been generalized to the concept of {\em ranking pattern}, and it was also proved that any ranking pattern can be realized using a single voting profile~\cite{DS-load, DSvoting, Saari(dictionary)}. The latter 
result could be useful for studying non-knockout tournaments and trying to see whether they can be manipulated.

\section*{Acknowledgements}
We are grateful to the referees for their valuable input. This work was partially supported by ``Dependent Point Processes: Neuronal Networks, Voting Theory, and Ageing Properties" funded by La Sapienza University, Roma 
 RP123188F3FCA243. 

The second author also thanks the Department of Data Analysis and Mathematical Modelling, Ghent University, for its kind hospitality.

\bibliographystyle{abbrv}

\end{document}